\documentclass[12pt,a4paper]{amsproc}
\usepackage[english]{babel}
\usepackage{latexsym}
\usepackage{amsmath}
\usepackage{amssymb}
\usepackage{amscd}
\usepackage[cp850]{inputenc}
\usepackage[mathscr]{eucal}
\theoremstyle{plain}
\newtheorem{quest}{Question}[section]

\newtheorem{prob}[quest]{Problem}
\newtheorem{defin}[quest]{Definition}

\newtheorem{thm}[quest]{Theorem}
\newtheorem{prop}[quest]{Proposition}
\newtheorem{corollary}[quest]{Corollary}
\newtheorem{cor}[quest]{Corollary}
\newtheorem{lemma}[quest]{Lemma}

\theoremstyle{remark}
\newtheorem{example}[quest]{Example}
\newtheorem{remark}[quest]{Remark}

\newcommand{\ideal}{u}

\newcommand{\Ideal}{U}
\newcommand{\R}{\mathbb{R}}
\newcommand{\N}{\mathbb{N}}
\newcommand{\K}{\mathbb{K}}
\newcommand{\C}{\mathbb{C}}
\newcommand{\T}{\mathbb{T}}

\newcommand{\Z}{\mathbb{Z}}

\newcommand{\adef}{\begin{defin}}
\newcommand{\zdef}{\end{defin}}
\newcommand{\algebra}{\mathcal A}
\newcommand{\Space}{{\rm Space}}
\newcommand{\In}{{\rm In}}

\newcommand{\Imp}{{\rm Imp}}

\newcommand{\Op}{{\rm Op}}
\newcommand{\op}{{\rm op}}
\newcommand{\Opp}{{\rm Op^{<\omega}}}

\newcommand{\aproof}{\begin{proof}}
\newcommand{\zproof}{\end{proof}}

\bigskip

\bigskip

\title{There is no largest proper operator ideal}

\author{Valentin Ferenczi}

\address{Departamento de Matem\'atica, Instituto de Matem\'atica e
Estat\'\i stica, Universidade de S\~ao Paulo, rua do Mat\~ao 1010,
05508-090 S\~ao Paulo SP, Brazil  \\ and \newline
Equipe d'Analyse Fonctionnelle \\
Institut de Math\'ematiques de Jussieu \\
Sorbonne Universit\'e - UPMC \\
Case 247, 4 place Jussieu \\
75252 Paris Cedex 05 \\
France.}
\email{ferenczi@ime.usp.br}



\thanks{The author acknowledges the support of FAPESP, project 2016/25574-8,
and of CNPq, grant 303731/2019-2}

\begin{document}

\begin{abstract} 
An operator ideal is proper if the only operators of the form ${\rm Id}_X$ it contains have finite rank.
We answer a question posed by Pietsch in 1979 (\cite{P}) by proving that there is no largest proper operator ideal. 
Our proof is based on an extension of the construction by Aiena-Gonz\'alez (\cite{AG}, 2000),  of an improjective 
but essential operator on Gowers-Maurey's shift space $X_S$ (\cite{GM2}, 1997), through a new analysis of the algebra
of operators on powers of $X_S$.

We also prove that certain properties hold for general $\C$-linear
operators if and only if they hold for these operators seen as real: for example this holds for the ideals of
strictly singular, strictly cosingular, or inessential operators,  answering a question of Gonz\'alez-Herrera (\cite{GH}, 2007). 
This gives us a frame to extend the negative answer to the question of Pietsch to the real setting.
\end{abstract}

\maketitle

\tableofcontents

\section{Introduction}

In this paper we consider operator ideals (or more generally, families of operators) in the sense of Pietsch \cite{P}. 
Unless specified otherwise by space we mean infinite dimensional Banach space and by subspace we mean closed infinite
dimensional subspace. An operator will be a bounded linear operator between Banach spaces, and $L(X,Y)$ denotes
the space of operators between the spaces $X$ and $Y$. If $\Ideal$ is an operator ideal, then $\Ideal(X,Y)$ is the
subset of operators of $L(X,Y)$ belonging to $\Ideal$.
For all other unexplained notation see what follows.

In his book Pietsch  considers a family of spaces associated to an ideal $\Ideal$, see \cite{P} 2.1: the {\em space ideal}  $\Space(\Ideal)$
defined by
$$X \in \Space(\Ideal) \Leftrightarrow Id_X \in \Ideal.$$

Of course this definition makes sense even when $\Ideal$ is a family of operators which is not an ideal.
Note that it is immediate that the space ideal of 
$\Ideal$ coincides with the space ideal of its closure $\Ideal^{\rm clos}$, \cite{P} Proposition 4.2.8, so 
in this context
one does not need to pay attention to whether the ideals considered are closed.
In \cite{P} 2.3.3 an ideal $\Ideal$ is called {\em proper} if $\Space(\Ideal) = F$, the class of finite dimensional spaces; or equivalently, if $\Ideal(X)$ is a proper ideal of $L(X)$ whenever $X$ is infinite dimensional.
Among proper ideals one can mention the ideals of finite rank, compact, strictly singular, strictly cosingular, or inessential operators,
see definitions below.
Problem \cite{P} 2.3.6 asks whether there is a largest proper operator ideal.
It is actually conjectured by Pietsch that such an ideal exists and is equal to the ideal $\In$ of inessential operators
(a specific case of \cite{P}  Conjecture 4.3.7, see \cite{P} 4.3.1).

\

\begin{prob}[Pietsch, 1979]\label{prob} Is the ideal of inessential operators the largest proper operator ideal?\end{prob}

\begin{prob}[Pietsch, 1979]\label{probbis} More generally, does there exist a largest proper operator ideal?
\end{prob}

\

This can also be seen as a special case of  \cite{P} Problem 2.2.8, where Pietsch asks whether, given a space ideal $A$
(see \cite{P} Definition 2.1.1), 
there exists a largest operator ideal $U$ with $A = \Space(U)$. Problems \ref{prob} and \ref{probbis} correspond to the space ideal
$F$ of finite dimensional spaces for which $F=\Space(\In)$.

\

Recall that an operator $T \in L(X,Y)$ is said to be {\em inessential}, $T \in \In(X,Y)$, if $Id_X-UT$ is Fredholm for any
$U \in L(Y,X)$ (equivalently $Id_Y-TU$ is Fredholm for any  such $U$); otherwise we shall say that it is essential. Two spaces
$X$ and $Y$ are {\em essentially incomparable} if $L(X,Y) = \In(X,Y)$; equivalently, $L(Y,X) = \In (Y,X)$.




There is a natural direction in which to investigate  whether $\In$ is the largest proper operator ideal, which
was suggested to the author by Manuel Gonz\'alez.
This would be to study the question in the setting of complex spaces as well as real spaces and obtain strong structural differences
between the complex and the real cases. Indeed if some $\C$-linear operator is essential as real but inessential as complex,
then this might mean that one gets a larger proper ideal than the complex ideal of inessential operators.

More generally it is a natural question, related to the study of complex structures on real Banach spaces,
to understand the differences between real and complex versions of some classical operator ideals,
and this is a first aim of this paper. More precisely we ask 
 whether a $\C$-linear operator belongs to a certain ideal as  $\C$-linear
if and only if it does as an $\R$-linear operator. It is obvious for example that an operator is compact as $\C$-linear if and only if
it is compact as $\R$-linear.
The question for strictly singular appears in \cite{GH} as Remark 2.7,
and for inessential was personally asked by M. Gonz\'alez.

While an $\R$-singular (resp. $\R$-inessential), $\C$-linear operator is clearly always $\C$-singular (resp. $\C$-inessential), the converse
is not immediate, since 
there are more real subspaces (resp. operators) than complex subspaces (resp. operators) in a complex space.
However we shall show that the answer is actually positive, and holds also for many other classical ideals.
The result depends on a characterization
based on the notion of {\em self-conjugacy} of a complex ideal, see Proposition \ref{11}.

\begin{thm}\label{1}  A $\C$-linear operator is inessential as a complex operator if and only if it is inessential as a  real operator. The same holds for the ideals of
\begin{itemize}
 \item 
strictly singular operators,
\item
strictly cosingular operators, 
\item $A$-factorable operators if $A$ is a complex and self-conjugate space ideal.
\end{itemize}
\end{thm}


\

Going back to Pietsch's problem, in particular the direction suggested above does not work.
In a second part of the paper we use another approach to  Problems \ref{prob} and \ref{probbis}, which we shall actually solve negatively.

\

An operator is {\em improjective}, $T \in \Imp(X,Y)$, if the restriction of $T$ to a complemented subspace of $X$ is never
an isomorphism onto a complemented subspace of $Y$, see Tarafdar \cite{Tarafdar}. When $L(X,Y)=\Imp(X,Y)$ (equivalently $L(Y,X)=\Imp(Y,X)$),
then $X$ and $Y$ are said to be {\em projectively incomparable}. It is straightforward that all inessential operators
are improjective, and that $Id_X$ is never improjective for $X$ infinite dimensional.


\

In 2000, Aiena and Gonz\'alez proved that there exist operators which are improjective but not inessential, \cite{AG} Theorem 3.6. 
Actually they obtain two projectively incomparable spaces and an operator between them which is essential, \cite{AG} Proposition 3.7. 
This suggests a direction to find a proper ideal larger than $\In$, providing a negative answer to Problem \ref{prob}: since $Id_X \in \Imp$
only when $X$ is finite dimensional, we would be done if $\Imp$ were an operator ideal.
However in the same paper Aiena and Gonz\'alez  prove that the improjective operators do not form an ideal, \cite{AG} Theorem 3.6.

The example of \cite{AG} relies on the theory of  spaces with few operators (or exotic spaces) of Gowers-Maurey, see \cite{maureyhandbook}.
As commented in the Aiena-Gonzalez paper, while hereditarily indecomposable spaces (first defined by Gowers-Maurey \cite{GM}) have the
property that all operators are either Fredholm or inessential,
on the other hand, in indecomposable spaces operators are either Fredholm or improjective; so it is natural to consider
an indecomposable space which is not HI. Their example is therefore based on the ``shift space" $X_S$ of Gowers-Maurey \cite{GM2}
which has these properties, see also Maurey's surveys \cite{maureyhandbook} and \cite{maurey} for a more thorough description.
Considering the complex version of $X_S$, they find a infinite codimensional subspace $Y$ of $X_S$
which is projectively incomparable with $X_S$; however there is an operator $T \in L(X_S,Y)$ which is not inessential.

If $X$ is a Banach space, $\Op(X)$ denotes the family of $X$-factorable operators.
This is an ideal if, e.g., $X$ is isomorphic to its square. It is easy to see that two spaces $X$ and $X'$
are projectively incomparable if and only if  $\Op(X) \cap \Op(X')$ is proper. So in particular
$\Op(X_S) \cap \Op(Y)$ is proper and contains an operator which is not inessential. A negative answer to Problem \ref{prob}
would follow if $\Op(X_S) \cap \Op(Y)$ were an ideal; but since $X_S$ is not isomorphic to its square this has no
reason to hold.

In this paper we show how to enhance Aiena-Gonz\'alez's result so that the associated $\Op$-class is an ideal: 
we define $\Op^{<\omega}(X)$ the class  of operators which are $X^n$-factorable for some
  $n \in \N$ and observe that it is an ideal. The crucial point is then to go back to the construction of \cite{GM2} to
  prove that all powers of the spaces
$X_S$ and $Y$ (or possibly some technical variation of them) are projectively incomparable, which means that
$\Ideal:=\Op^{<\omega}(X_S) \cap \Op^{<\omega}(Y)$ is a proper ideal. Since the essential operator $T$ defined in \cite{AG} belongs
to $\Ideal$, the ideal of inessential operators is not the largest among proper ideals. 
This answers Question \ref{prob} of Pietsch.

\

Based on the observation of Aiena-Gonz\'alez that their construction actually provides
an example of two improjective operators whose sum is not improjective,  we find two versions
of the above ideal and two operators belonging to each of them but whose sum is invertible on
$X_S$. As
a corollary there actually cannot exist a largest proper ideal. So we have a stronger result,
namely the answer
to Question \ref{probbis} of Pietsch is also negative.
\begin{thm}\label{2} There is no largest complex proper operator ideal. \end{thm}

These examples hold in the complex setting. We use some ideas of the first part of the paper to extend our negative answers to the real setting as well.
To be able to treat both the complex and real cases in a unified way, we shall replace the complex version (call it $X_S(\C)$) of $X_S$
used in the above description, by the complexification $X:=(X_S(\R))_\C$ of the real version $X_S(\R)$ of $X_S$. 
While these two spaces are certainly not isomorphic,
their algebras of operators have very similar properties, sufficiently for our purposes, and so all of the above applies
to $X$. But additionally $X$ is much easier to relate
to a real space (through complexification), and this will provide us with
a real solution based on operators on $X_S(\R)$.

\begin{thm}\label{3} There is no largest real proper operator ideal. \end{thm}


\subsection{Background and definitions}

\

In what follows $I_X$, or sometimes $Id_X$, denotes the identity map on $X$. We use the notation $X \simeq Y$ to mean
that the spaces $X$ and $Y$ are linearly isomorphic.

We recall a few basic results about certain operator ideals and Fredholm theory. For more details we
refer to \cite{LT} or to the survey of Maurey \cite{maurey}.

An operator $S \in L(X,Y)$ is
strictly singular, $S \in SS(X,Y)$, when $S_{|Z}$ is never an isomorphism into  for $Z$ an (infinite dimensional) subspace of $X$;
it is strictly cosingular, $S \in CS(X,Y)$,
when $QS$ is never surjective for $Q$  the quotient map onto some infinite codimensional subspace of $Y$).

An operator $T: X \rightarrow Y$ is Fredholm if it has closed image and finite dimensional kernel and cokernel.
It is finitely singular if it restricts to an isomorphism into on some finite codimensional subspace -
this terminology appears in \cite{GM2}; such operators are more classically called upper semi-Fredholm, as in \cite{AG}.
It is infinitely singular otherwise, which is equivalent to saying that for any $\epsilon>0$ there exists an
infinite dimensional subspace $Z$ of $X$ such that $\|T_{|Z}\|$ is at most $\epsilon$.

Recall that $K$ denotes the closed ideal of compact operators. We have the following classical inclusions:
$$K(X,Y) \subset SS(X,Y) \subset In(X,Y) \subset Imp(X,Y)$$
and the chain of inclusions obtained by replacing $SS(X,Y)$ by $CS(X,Y)$.

The ideal of inessential operators is closely related to Fredholm theory; in particular
an inessential perturbation of a Fredholm operator is Fredholm (and so this holds as well for
compact or strictly singular perturbations).

A Banach space is decomposable if it is the (topological) direct sum of two infinite dimensional closed subspaces, indecomposable
otherwise, and hereditarily indecomposable (HI) if it contains no decomposable subspace.
The first example of an HI space was due to Gowers-Maurey \cite{GM} and since then a great number
of other indecomposable or HI examples with various additional properties have been obtained (some of which may be found in 
\cite{maureyhandbook}).

\section{Complex ideals versus real ideals}

In this section we recall and develop tools to compare $\R$-linear and $\C$-linear behaviours of operators,
with Theorem \ref{1} as our objective.

\subsection{Complex structures}

\

The theory of complex structures on Banach spaces was born after the example by Bourgain (1986)
of two spaces which are linearly isometric
as real spaces but not isomorphic as complex spaces \cite{Bourgain}. Actually the two spaces used by
Bourgain are conjugate and so the real linear isometry is just the identity map between them.

A complex structure on a real space $X$ is the space $X$ equipped with a $\C$-linear structure whose underlying real structure coincides
with the original one. Allowing renormings, this is in correspondence with real operators $J$ on $X$ of square equal to $-I_X$, which define
the multiplication $x \mapsto i.x$.
The {\em number of complex structures} on a space is understood up to ($\C$-linear) isomorphism and has been studied in several papers.
For example a real space is said to have {\em unique complex structure} if it admits  complex structures and all of them are mutually isomorphic.
Examples of spaces with unique complex structure are: a) the Hilbert space (folklore or the next list of examples),
b) the spaces $\ell_p, L_p(0,1), c_0, C([0,1])$ and more generally real spaces 
admitting a complex structure and whose
complexification is primary (Kalton, Theorem 28 in \cite{FG}),  c) some hereditarily indecomposable example \cite{F}, d) some non-classical example with a subsymmetric basis
\cite{C}, and e) others.
Examples of spaces without complex structure are James space \cite{dieudonne}, a uniformly convex space of Szarek \cite{szarek},
the original Gowers-Maurey space \cite{GM}, as well as many other spaces
with small spaces of operators. ``Extremely non-complex" real spaces are also considered in \cite{koszmider}.

In \cite{F} are also provided spaces with exactly $n$ complex structures, whenever $n \geq 2$.
This also gives examples of spaces with a complex structure which is not unique but still is isomorphic to its conjugate. 
An example with exactly $\aleph_0$ complex structures is due to Cuellar \cite{cuellar}, and
one with $2^{\aleph_0}$ and additional properties is due to Anisca \cite{anisca} (it is not hard to check that the original example of
Bourgain also admits $2^{\aleph_0}$ such structures).
See also \cite{AFM} for considerations on the number of complex structures in the setting of complexity of equivalence
relations on Polish spaces.

In \cite{K} Kalton uses a variation of Kalton-Peck space $Z_2$ from \cite{KP}, called $Z_2(\alpha)$ ($\alpha$ a non-zero real parameter), 
to obtain a much simpler example of complex space  not isomorphic to its conjugate $\overline{Z_2(\alpha)}$ (which
here identifies with $Z_2(-\alpha)$). According 
to the proof of \cite{K} Theorem 2, see \cite{mestradocuellar}, it actually holds that $Z_2(\alpha)$ does not even embed
into $\overline{Z_2(\alpha)}$. Regarding $Z_2$ 
it seems to be an interesting open question whether it admits a unique complex structure.
Finally the most extreme example seems to appear in  \cite{F}, with a space admitting
exactly two complex structures, which are conjugate (and therefore $\R$-linearly isometric) but  totally incomparable as complex spaces (meaning
that no $\C$-linear subspace of one is $\C$-isomorphic to a $\C$-linear subspace of the other).

These examples show that there can be quite a variety of complex structures on a given real space, and therefore
it is a natural and non trival question not only to relate properties of operators seen as real or seen as complex,
but also seen as $\C$-linear with respect to different complex structures on the same real space.

\

We refer to Pietsch \cite{P} for background on operator ideals. In this paper we
shall use the word {\em class} to define a family of normed spaces which is stable under isomorphisms.
A {\em class} of operators which does not necessarily define an ideal is also defined in the sense of Pietsch,
i.e. with varying domain and codomain.

The concept of complexification of real spaces, and of real operators on them,  is well-known, and recalled below. 
It is for example extremely useful in order to use
spectral theory  in the context of real spaces. There is a less well-known and almost trivial process, which we shall call here {\em realification},
and which
is simply the one obtained by ``forgetting" the multiplication by $i$ on a space and ``only remembering" the $\R$-linear structure.

We list the definitions of complexification and realification in various situations below.
Before that, let us fix an important notation. Since we shall always go back and forth between real and complex
ideals or classes, to avoid confusion and when relevant we shall reserve lower case letters ($u$, $ss$, $cs$, $in$, ...) for classes of
{\em real} operators and upper case letters ($U$, $SS$, $CS$, $IN$, ....) for classes of {\em complex} operators.
The same will hold for classes of spaces ($a$,... for classes of real spaces, $A$,... for classes of complex spaces).

\subsection{Normed spaces}
The complexification $X_\C$ of a real space $X$ is the space $X \oplus X$ equipped with the complex structure
associated to $J(x,y)=(-y,x)$. Elements of $X_\C$ are often noted $x+iy, x,y \in X$. Regarding the realification:

\begin{defin} Let $X$ be a complex space. The realification $X_\R$ of $X$ is the space $X$
equipped with the real structure underlying its complex structure.
\end{defin}

As is usual we denote by $\overline{X}$ the conjugate of the complex space $X$, i.e. the space $X$
equiped with the law $\lambda . x:=\overline{\lambda}x$.
 It is clear that the realifications of $X$ and $\overline{X}$ coincide.
Note also that if $T$ is $\C$-linear from $X$ to $Y$, then it also acts as a $\C$-linear operator,
denoted $\overline{T}$, from $\overline{X}$ to $\overline{Y}$.

\begin{remark}\label{rem} The following hold:
\begin{itemize}
\item[(a)] if $X$ is a real space then $(X_\C)_\R=X \oplus X$.
\item[(b)] if $X$ is a complex space then
$(X_\R)_\C \simeq X \oplus \overline{X}$.
\end{itemize}
\end{remark}

\begin{proof} The first is obvious. The second follows from an observation by N.J. Kalton which appears
in a first form in \cite{FG} Lemma 27
and then more clearly in a paper of W. Cuellar Carrera \cite{C} Lemma 2.1. Namely for any real space $X$ and any complex structure $J$ on $X$,
denoting by $X_J$ the associated complex case, we have
$$X_\C \simeq X_J \oplus X_{-J}$$ as complex spaces.
The copy of $X_J$ in $X_\C$ is the subspace $\{(x,Jx); x \in X\}$, or more explicitely the isomorphism from
$X_J \oplus X_{-J}$ to $X_\C$ is given by
$$(x,y) \mapsto (x,Jx)+(y,-Jy)=(x+y, J(x-y)).$$
\end{proof}

\subsection{Classes of spaces}
It is then natural to define complexification and realification of classes of spaces, where we recall that the classes
are understood to be invariant by isomorphism.
\begin{defin}

\

If $a$ is a class of real spaces, we define the class $a_\C$ of complex spaces by
$$X \in a_\C \Leftrightarrow X_\R \in a.$$

If $A$ is a class of complex spaces, we define the class $A_\R$ of real spaces by
$$X \in A_\R \Leftrightarrow X_\C \in A.$$
\end{defin}

\begin{remark} The following hold:
\begin{itemize}
\item[(a)]
If $X$ is a real space and $a$ a class of real spaces, then
$X \in (a_\C)_\R$ iff $X^2 \in a$.
\item[(b)]
If $X$ is a complex case and $A$ a class of complex spaces, then $X \in (A_\R)_\C$ iff $X \oplus \overline{X} \in A$.
\end{itemize}
\end{remark}

\subsection{Linear operators}
Similar concepts are defined for bounded linear operators.

\begin{defin} 

\

If $T$ is real from $X$ to $Y$ then its {\em complexification} $T_\C$ from $X_\C$ to $Y_\C$ is well-known, and defined as
$$T_\C(x+iy)=Tx+iTy.$$ Conversely for
$T$ $\C$-linear between complex spaces $X$ and $Y$, its {\em realification} $T_\R$ will be $T$ seen as $\R$-linear between $X_\R$ and $Y_\R$.
\end{defin}
\

Note  that
$T \mapsto T_\C$ is an algebra homomorphism from ${\mathcal L}(X,Y)$ to
${\mathcal L}(X_\C, Y_\C)$, and  that
$T \mapsto T_\R$ is an algebra homomorphism from ${\mathcal L}(X,Y)$ to
${\mathcal L}(X_\R, Y_\R)$.
As a consequence:

\begin{remark} 
The following hold:
\begin{itemize}
\item[(a)] if $T$ is  $\R$-linear then the realification of the complexification of $T$ is 
$\begin{pmatrix} T & 0 \\ 0 & T \end{pmatrix}$ acting from $X^2$ to $Y^2$
\item[(b)] if $T$ is $\C$-linear then the complexification of the realification of  $T$ may be seen as
$\begin{pmatrix} T & 0 \\ 0 & T \end{pmatrix}$ acting from $X \oplus \overline{X}$ to $Y \oplus \overline{Y}$
\end{itemize}
\end{remark}

\begin{proof} We just note that $(T_\R)_\C$ defined on $(X_\R)_\C$ acts on the copy $\{(x,ix), x \in X\}$ of $X$ as $$(x,ix) \mapsto (Tx,Tix)=(Tx,iTx).$$
This means that it acts as $T$ on $X$ modulo the identification of $X$ as a subspace of $(X_\R)_\C$ through $x \mapsto (x,ix)$.
Likewise it acts as $T$ on $\overline{X}$ modulo the identification of $\overline{X}$ as a subspace of $(X_\R)_\C$ through $x \mapsto (x,-ix)$.
\end{proof}

\subsection{Classes of operators and/or ideals}

Finally we define complexification and realification for classes of operators. We shall see 
that these definitions behave well with ideals in the sense of Pietsch.

\begin{defin} 

\

\begin{itemize}
\item[(a)]
Let $\ideal$ be a class of real operators. We define the complexification $\ideal_\C$ of $\ideal$ by
$$T \in \ideal_\C \Leftrightarrow T_\R \in \ideal$$
\item[(b)]
Let $\Ideal$ be a class of complex operators. We define the realification $\Ideal_\R$ of $\Ideal$ by
$$T \in \Ideal_\R \Leftrightarrow T_\C \in \Ideal$$
\end{itemize}
\end{defin}

\begin{lemma}
If $\ideal$ is a real (closed) ideal of operators then $\ideal_\C$ is a complex (closed) ideal. 
If $\Ideal$ is a complex (closed) ideal of operators then $\Ideal_\R$ is a real (closed) ideal.
\end{lemma}

 For $\ideal_\C$ note that this relies on the fact that if $T \in \ideal_\C$ then $iT \in \ideal_\C$, 
 because $(iT)_\R=iT_\R \in \ideal$ since $i$ is an $\R$-linear operator and $\ideal$ is a real ideal.



The following natural notion will prove extremely important.
\subsection{Conjugate classes and/or ideals}

\begin{defin}
For $\Ideal$ a complex class of operators  let us denote
by $\overline{\Ideal}$ the conjugate class, i.e.
$$T \in \overline\Ideal \Leftrightarrow \overline{T} \in \Ideal.$$
\end{defin}

\begin{defin} A complex class $\Ideal$ of operators is self-conjugate if $\overline\Ideal=\Ideal$.\end{defin}

The class $\overline{\Ideal}$ is not to be mistaken with the closure of $\Ideal$, which is denoted
$\Ideal^{\rm clos}$ . The proof of the next proposition is left as an exercise.

\begin{prop}
The ideals of compact, strictly singular, strictly cosingular, inessential operators,
and the class of improjective operators are self-conjugate. 
\end{prop}

\begin{prop}
If $\ideal$ is a real class of operators, then $\ideal_\C$ is self-conjugate.
\end{prop}

\begin{proof} For a complex operator $T$ the real operators $T_\R$ and $\overline{T}_\R$ coincide.
\end{proof}

\

To develop examples of ideals which are not self-conjugate, we consider
$\Op(X)$, the class of $X$-factorable operators, i.e. operators which factor through the Banach space $X$. 

\begin{defin}\label{defop} If $X$ is a Banach space, then $\Op(X)$ denotes the class of $X$-factorable operators, i.e.
 for $T \in L(Y,Z)$,
 $T \in \Op(X)$ iff $T=UV$ for some $V \in L(Y,X)$ and $U \in L(X,Z)$.
\end{defin}

Let us note the useful observation
that when $X$ is a complex space,  $\overline{\Op(X)}=\Op({\overline{X}})$.
We recall the well-known fact:

\begin{prop} If $X$ is a Banach space which contains a complemented subspace isomorphic to $X^2$, then $\Op(X)$ is an operator ideal. 
\end{prop}

Note that $\Op(X)$ has no reason to be closed in general.

\begin{prop}\label{notself} Let $X$ be a complex space which is not isomorphic to a complemented subspace of $\overline{X}$.
Then $\Op(X)^{\rm clos}$ is not self conjugate. In particular $\Op(X)$ is not self-conjugate.
\end{prop}

\begin{proof} 
We shall prove that $I_X$ does not belong to $\overline{\Op}(X)^{\rm clos}=\Op(\overline{X})^{\rm clos}$.

Indeed assume there exists $A,B$ such that $T:=I_{X}-AB$ has norm $\|T\|<\epsilon$ where $B: X \rightarrow \overline{X}$ and $A: \overline{X} \rightarrow X$.
Then for $\epsilon$ small enough
$AB=I-T$ would be an isomorphism on $X$  and
therefore $B$ would be an isomorphic embedding of $X$ into $\overline{X}$. 
Finally the image $BX$ would be complemented 
in $\overline{X}$ by $B(I-T)^{-1}A$. This is a contradiction.
\end{proof}

Of course spaces not isomorphic to a complemented subspace of their conjugate 
and at the same time isomorphic to their squares (so that $\Op(X)$ is an ideal) must be rather exotic.
We present two examples of such spaces and therefore of ideals which are not self-conjugate.

\begin{example} If $F$ is the complex HI space totally incomparable with its conjugate from \cite{F},
then the ideal $\Op(\ell_2(F))^{\rm clos}$ is not self conjugate.\end{example}

\begin{proof} The space $F$ is complemented in $\ell_2(F)$ but does not embed in $\overline{\ell_2(F)}=\ell_2(\overline{F})$.
Indeed, see for example \cite{Burlando}, a space which embeds into $\ell_2(\overline{F})$ either contains a copy of $\ell_2$ (which cannot hold
in the case of the HI space $F$) or embeds into $\overline{F}^n$ for some $n$, which contradicts the total incomparability
of $F$ with $\overline{F}$.  \end{proof}

A less exotic example, even ``elementary" in the words of Kalton, is provided by him in \cite{K}.

\begin{example}
If $Z_2(\alpha)$ is the version of Kalton-Peck complex space  defined by Kalton \cite{K}, 
then  $\Op(Z_2(\alpha))^{\rm clos}$ is an ideal which is not self conjugate, for $\alpha \neq 0$. \end{example}

\begin{proof} The space $Z_2(\alpha)$  does not embed into its conjugate, if $\alpha \neq 0$, see \cite{K} Proof of Theorem 2
and \cite{mestradocuellar}. On the other hand, it admits a canonical $2$-dimensional ``symmetric decomposition" in the same way as
$Z_2$ does and in particular
is isomorphic to its square.
\end{proof}

\section{Applications to real and complex versions of ideals}

\subsection{Real and complex versions of classical ideals}

We use the analysis of the previous section to relate a certain correspondence between
real and complex versions of ideals to the self-conjugacy property.

\begin{prop}\label{11}

\

\begin{enumerate}
\item
Let $\ideal$ be a real ideal. Then
$(\ideal_\C)_\R=\ideal.$
\item
Let $\Ideal$ be a complex ideal. Then
$(\Ideal_\R)_\C = \Ideal \cap \overline{\Ideal}.$
\item
A complex ideal $\Ideal$  is self-conjugate if and only if
$(\Ideal_\R)_\C = \Ideal.$
\end{enumerate}
\end{prop}

\begin{proof} (1)
Indeed $T \in (\ideal_\C)_\R$ if and only if
$T_\C \in \ideal_C$ if and only if $(T_\C)_\R \in \ideal$, which means that
$\begin{pmatrix} T & 0 \\ 0 & T \end{pmatrix}$ belongs to $\ideal$ and is equivalent to $T \in \ideal$ by the ideal properties.

(2) $T \in (\Ideal_\R)_\C$ if and only if
$T_\C \in \Ideal_C$ if and only if $(T_\R)_\C \in \Ideal$, which means that
$\begin{pmatrix} T & 0 \\ 0 & T \end{pmatrix}$ acting on $X \oplus \overline{X}$ belongs to $\Ideal$; 
this is equivalent to $T, \overline{T} \in \Ideal$ by the ideal properties.

(3) follows immediately.
\end{proof}

\

We shall write about the real and complex versions of the ideals of  strictly singular, strictly cosingular,
inessential operators, and class of improjective operators.
We denote $ ss, cs, in, imp$ the real version and $SS, CS, IN, IMP$ the complex version of these.

Let us first note that a $\C$-linear operator is $\C$-strictly singular as soon as it is $\R$-singular.
In our language
$$ss_\C \subset SS.$$
It is an easy exercise that the property $$u_\C \subset U$$ also
holds if $u=cs, in, imp$ and $U=CS, IN, IMP$, respectively.
Actually we have

\begin{prop}
A real map $T$ is strictly singular (resp. strictly cosingular, inessential, improjective)
if and only if $T_\C$ is strictly singular (resp. strictly cosingular, inessential, improjective). In other words,
$$U_\R=u$$ 
holds if $u=cs, in, imp$ and $U=CS, IN, IMP$, respectively.
 
\end{prop}

\begin{proof} We use Proposition \ref{11}. Since $ss_\C \subset SS$ then $ss=(ss_\C)_\R \subset SS_\R$. Conversely if $T: X \rightarrow Y$ is not singular,
let $Z \subset X$ be such that $TI_{Z,X}$ is an isomorphism into $Y$. Then
$T_\C I_{Z_\C,X_\C}$ is a $\C$-linear isomorphism from $Z_\C$ into $Y_\C$ and since $Z_\C$ is a $\C$-linear subspace
of $X_\C$, $T_\C$ is not strictly singular. Summing up $T \notin ss \Rightarrow T \notin SS_\R$.

Since $cs_\C \subset CS$, the inclusion $cs \subset CS_\R$ holds. Conversely if $T: X \rightarrow Y$ is not cosingular, let $Z \subset Y$ be infinite
codimensional such that
$QT$ is surjective, where $Q$ is quotient map from $Y$ onto some $Z$.  Then
$Q_\C$ is the quotient map from $Y_\C$ onto $Z_\C$ and $Q_\C T_\C$ is surjective, therefore $T_\C$ is not cosingular.

Since $in_\C \subset IN$, the inclusion $in \subset IN_\R$ holds. Conversely if $T: X \rightarrow Y$ is not inessential,
let $U: Y \rightarrow  X$ be such that $Id-UT$ is not Fredholm. Then
$(Id-UT)_\C=Id-U_\C T_\C$ is not Fredholm, and therefore $T_\C$ is not inessential.

Since $imp_\C \subset IMP$, the inclusion $imp \subset IMP_\R$ holds. Conversely if $T: X \rightarrow Y$ is not improjective,
let $W$ be complemented in $X$ and $Z$ in $Y$ such that $T$ restricts to an isomorphism between $W$ and $Z$.
Then $T_\C$ restricts to an isomorphism between the complemented subspaces $W_\C$ and $Z_\C$ of $X_\C$ and $Y_\C$ respectively,
so is not inessential.
\end{proof}

\begin{corollary}
 A $\C$-linear operator is strictly singular (resp. strictly cosingular, inessential) if and only if it is strictly singular
 (resp. strictly cosingular, inessential) as $\R$-linear.
In other words $$U=u_\C$$ holds if $u=cs, in, imp$ and $U=CS, IN, IMP$, respectively.
\end{corollary}

\begin{proof} Since $ss=SS_\R$, it follows that $ss_\C=(SS_\R)_\C$ and this is equal to $SS$ by Proposition \ref{11},
since $SS$ is self-conjugate. The same reasoning holds for cosingular and inessential operators. \end{proof}

We formalize these ideas as follows:

\begin{prop}\label{equiv} Let $\Ideal$ be a  complex ideal,  and let $\ideal={\Ideal}_\R$, i.e.,
$T \in \ideal \Leftrightarrow T_\C \in \Ideal.$
Then the following are equivalent:
\begin{itemize} 
\item[(a)] for any  complex operator $T$ between two complex spaces, $T \in \Ideal$ if and only if $T$ seen as real  is in $\ideal$,
 \item[(b)] $\ideal_\C=\Ideal$,
\item[(c)]  $\Ideal$ is self-conjugate. 
\end{itemize}
\end{prop}

\begin{defin} When $\ideal=\Ideal_\R$ and (a)-(b)-(c) of Proposition \ref{equiv} hold, 
we say that $(\ideal, \Ideal)$ is a {\em regular} pair of ideals.
\end{defin}

\begin{cor} The pairs $(ss,SS)$, $(cs,CS)$, and $(in,IN)$ are regular.\end{cor}

\

In terms of complex structures on a real Banach space, 
this also means:

\begin{cor} If $(\ideal,\Ideal)$ is a regular pair of ideals,  then
an operator belonging to $\Ideal$ with respect to a complex structure on the real space $X$,
  also belongs to $\Ideal$ with respect to any other complex structure on $X$ for which it is $\C$-linear.
\end{cor}

Another very relevant family of operator ideals are the ideals $\Op(A)$, generalizing Definition \ref{defop} of $\Op(X)$. 
According to \cite{P} Definition 2.1.1 a {\em space ideal $A$} is a class of spaces containing the finite dimensional ones and stable
under taking direct sums and complemented subspaces. The ideal $\Op(A)$ is defined in \cite{P} 2.2.1:

\begin{defin} If $A$ is a space ideal, then
$T \in \Op(A)$ if and only if $T$ is $X$-factorable for some $X \in A$. 
\end{defin}

If $A$ is a complex space ideal we define in an obvious may the conjugate space ideal $\overline{A}$ by
$$X \in \overline{A} \Leftrightarrow \overline{X} \in A,$$
and say that $A$ is self-conjugate if $A=\overline{A}$.
 If $A$ is complex, we also denote by $\op(A)$ the ideal of real operators which factor (by $\R$-linear operators) through some
$X \in A$ (seen as real). 

\begin{prop}\label{pliz} Let $A$ be a complex and self-conjugate space ideal.
Then $(\op(A),\Op(A))$ is a regular pair of ideals.
 \end{prop}

 \begin{proof} We claim that $\op(A)=\Op(A)_\R$. Indeed assume $T$ is a real operator factoring trough $X \in A$.
  Then $T_\C$ factors through $X_\C$. Since $X_\C$ is isomorphic to $X \oplus \overline{X}$ by Remark \ref{rem}(b),
  and since $A$ is a self-conjugate space ideal, it also belongs to
  $A$. So $T_\C$ belongs to $\Op(A)$, which means by definition that $T$ belongs to $\Op(A)_\R$. Conversely
  if $T_\C$ belongs to $\Op(A)$, then the matrix $\begin{pmatrix} T & 0 \\ 0 & T \end{pmatrix}$ belongs to $\op(A)$
  from which it follows easily that $T$ itself belongs to $\op(A)$.
  Since the claim holds, the result follows from the fact that $\Op(A)$ is obviously self conjugate and from
  Proposition \ref{equiv}.
 \end{proof}

 \
 
 The above extends obviously to ideals of operators $T \in L(Y,Z)$ which factorize through $A$ as  operators
 of $L(Y,Z^{**})$. As easy  application we also obtain the
regular pair of ideals: (real $\ell_p$-factorable operators, complex $\ell_p$-factorable operators),
(real $\sigma$-integral operators, complex $\sigma$-integral operators),..., see \cite{P} 19.3 e 23 for details.
We also leave as an exercise to the reader to find examples of regular pair of ideals related
to the ideal $\Ideal_T$ of operators factorizing through a given operator $T$ (under the necessary restrictions).


\subsection{Improjective operators and examples of non-regular pairs}
Since improjective operators do not form an ideal, according to \cite{AG}, Proposition \ref{equiv} does not apply to them.
What is true is the following slightly more restrictive statement:

\begin{prop}
Let $X, Y$ be two complex Banach spaces such that 
$Imp(X^2,Y^2)$ is a linear subspace of $L(X^2,Y^2)$, 
Then a $\C$-linear operator $T$ between $X$ and $Y$ is improjective if and only if it is improjective as $\R$-linear.
\end{prop}
\begin{proof} We already observed that $\R$-improjective implies $\C$-improjective. Assume now $T$ is  not improjective as $\R$-linear. 
Then $(T_\R)_\C$ is not improjective between $X_\C$ and $Y_\C$, which is the same as saying that 
$\begin{pmatrix} T & 0 \\ 0 & T \end{pmatrix}$ is not improjective from $X \oplus \overline{X}$ to $Y \oplus  \overline{Y}$, or equivalently
$\begin{pmatrix} T & 0 \\ 0 & \overline{T} \end{pmatrix}$ is not improjective from $X^2$ to $Y^2$. By the hypothesis, for example
$\begin{pmatrix} 0 & 0 \\ 0 & \overline{T} \end{pmatrix}$ is not improjective from $X^2$ to $Y^2$, which is the same as saying that $\overline{T}$
is not improjective from $X$ to $Y$, and this is equivalent to $T$ not improjective from $X$ to $Y$. 
\end{proof}

\

We can use the examples of non-self conjugate ideals from Section 2 to give immediate examples of pairs which are not regular,
showing that the hypotheses of Proposition \ref{pliz} are necessary.
Let $X$ be $Z_2(\alpha)$ or $\ell_2(F)$ and consider the real and complex ideals
$\Op(X),$
the ideal of $X$-factorable $\C$-linear operators (factorizing with $\C$-linear maps),
and
$\op(X),$
the ideal of $X$-factorable $\R$-linear operators (factorizing with $\R$-linear maps).
Then:

\begin{example} The pair $(\op(X),\Op(X))$ is not a regular pair.
\end{example}

\begin{quest} Find other relevant examples of regular or non-regular pairs of ideals. \end{quest}

\

\section{A solution to the problem of Pietsch}

Recall that an operator ideal (or class) $\Ideal$ is proper if $I_X \in \Ideal$ implies $X$ finite dimensional.
While $\Op(X)$ is the class of $X$-factorable operators,  we also define

\begin{defin} Let $X$ be a Banach space. We denote by
 $\Opp(X)$ the ideal of operators which are $X^n$-factorable for some $n \in \N$.
\end{defin}

It is clear that $\Opp(X)$ is an ideal: if $R, T \in L(Y,Z)$  are $X^m$ and $X^n$-factorable respectively,
then $R+T$ is $X^{m+n}$-factorable. See for example the proof of \cite{P} Theorem 2.2.2. Actually $\Opp(X)$ coincides
with $\Op(A_X)$, if $A_X$ is the space ideal of spaces which embed complementably into some power of $X$.
 
\begin{remark} 
Let $X, Y, Z$ be infinite dimensional Banach spaces. Then 
\begin{itemize}
\item[(a)] $I_Z \in \Op(X)$ if and only if $Z$ embeds complementably in $X$,
\item[(b)] $\Op(X) \cap \Op(Y)$ is proper if and only if $X$ and $Y$ are projectively incomparable,
\item[(c)] $\Opp(X) \cap \Opp(Y)$ is proper if and only $X^m$ and $Y^n$ are projectively incomparable
for all $m,n \in \N$.

\end{itemize}
\end{remark}
\begin{proof}
 (a) If $I_Z=UV$ is a factorization witnessing that $I_Z \in \Op(X)$,
 then $VU$ is a projection onto the isomorphic copy $VZ$ of $Z$.
  (b) the class $\Op(X) \cap \Op (Y)$ is not proper when there exists an infinite dimensional
  space $Z$ such that $I_Z \in \Op(X) \cap \Op(Y)$, i.e. by (a) $Z$ embeds complementably in $X$ and $Y$.
  (c) follows from (b) and the fact that $\Opp(X)=\cup_n \Op(X^n)$. \end{proof}

\

We now consider $X_S$, the ``shift-space" defined by Gowers and Maurey in \cite{GM2},  also \cite{maureyhandbook} and many details
in \cite{maurey} (see also \cite{bence} for considerations on equivalence of projections on $X_S$).  
The space $X_S$ is an indecomposable, 
non hereditarily indecomposable space, admitting a Schauder basis for which the shift operator $S$ is an isometric embedding, implying
that $X$ is isomorphic to its hyperplanes. Actually the complex version of $X_S$ has the very strong following rigidity property: 
\begin{prop}[Gowers-Maurey]\label{GM} The following are equivalent for a subspace $Y$ of $X_S$:
\begin{enumerate}
\item $Y$ is isomorphic to $X_S$
\item $Y$ is complemented in $X_S$
\item $Y$ is finite codimensional in $X_S$
\end{enumerate}
\end{prop}

We shall use the next crucial proposition, whose proof is postponed until the next section and is of a more technical nature.
The proof involves multidimensional versions of the machinery used by Gowers and Maurey in \cite{GM2}, and therefore requires some familiarity
with the use of $K$-theory for algebras of operators on Banach spaces and in particular properties of Fredholm operators, as quite well explained in
\cite{maurey}. It also requires certain facts of the $K$-theory of the
Wiener algebra $A(\T)$, as well as some conditions to apply complexification and obtain the real case. For these reasons we keep
those details for the next section.

\begin{prop}\label{mainprop} Let $X_S$ be the real or complex shift space of Gowers-Maurey. Assume $m, n \in \N$. 
Let $Y$ be an infinite codimensional subspace of $X_S$.
Then there is no isomorphism between a complemented subspace of $X_S^m$ and a
 subspace of $Y^n$.
\end{prop}

Let us note here that we shall actually prove that a complemented subspace of $X_S^m$ must be
isomorphic to $X_S^q$ for some $q \leq m$, and therefore Proposition \ref{mainprop} will follow from the fact
that $X_S$ does not embed into $Y^n$. Note also that the case $m=n=1$ in the complex case is immediate from Proposition \cite{GM}
and this is the idea that was used in \cite{AG}.

\

Let us first mimic the construction of \cite{AG} inside $X_S$.
Given $t \in \T$ (resp. $\{-1,1\}$ in the real case), the operator $Id-tS$ is injective.
We claim that its image is not closed; indeed otherwise $Id-tS$ would be an isomorphism
onto its image, and this is false, by considering for any $N \in \N$, the vector 
$$x_N=\sum_{n=1}^N t^n e_n,$$ which
has norm at least $n/\log_2(n+1)$ by \cite{GM2} Theorem 5, while
$$(Id-tS)(x_N)=e_0-t^{n+1} e_{n+1}$$ has norm at most $2$.
This implies that for any $t \in \T$ (resp. $\{-1,1\}$ in the real case) and for some compact operator $K_t$ on $X$, the operator
$$T_t:=Id-tT+K_t$$
has image of infinite codimension (see for example \cite{LS} Theorem 5.4). Denote $Y_t={\rm Im}(T_t)$.

\begin{prop} Given $t \in \T$ (resp.  $\{-1,1\}$ in the real case), the ideal  $\Ideal_{t}:=\Opp(X_S) \cap \Opp(Y_t)$ is a proper ideal
which is not contained in the ideal of inessential operators.
\end{prop}

\begin{proof} Since $Y_t$ is infinite-codimensional, by Proposition \ref{mainprop}, any powers of $X_S$ and of $Y_t$ are
projectively incomparable, or equivalently, $\Opp(X_S) \cap \Opp(Y_t)$ is a proper ideal.
 Denote by $i_{Y,X_S}$  the canonical inclusion of $Y$ inside $X_S$.
 The operator $T_t: X_S \rightarrow Y_t$ belongs to $\Ideal_t$, and it is essential, since 
 $Id-\frac{1}{2} i_{Y,X_S}T_t=I-\frac{1}{2}(I-tS+K_t)=\frac{1}{2}(I+tS-K_t)$ is not Fredholm.
\end{proof}

\begin{thm} There is no largest proper real or complex ideal.  \end{thm}

\begin{proof} An ideal $\Ideal$ containing all proper ideals msut contain $\Ideal_1$ and $\Ideal_{-1}$.
Therefore the operators $T_1=Id-S+K_1$ and $T_{-1}=Id+S+K_{-1}$ 
 belong to $\Ideal$ and the same holds for these operators seen as operators from $X_S$ to $X_S$. 
 Then the Fredholm operator $T_1+T_{-1}=2Id+K_1+K_{-1}$
 on $X_S$ belongs to $\Ideal$, and therefore $Id_{X_S}$ belongs to $U$. 
 Since $X_S$ is infinite dimensional, $\Ideal$ cannot be proper.
\end{proof}










\section{The proof of projective incomparability}

This section is devoted to the proof of Proposition \ref{mainprop}.

\subsection{Complex version versus complexification of the shift space}

\

We recall a few facts from \cite{GM2}. If $X_S(\K)$ is the version of the shift space defined on $\K=\R$ or $\C$, then
there exists an algebra homomorphism and projection map $\Phi$ from $L(X_S(\K))$ to some algebra of operators denoted $\algebra$.
$S$ denotes the right shift and $L$ the left shift on the canonical basis of $X_S$.
Elements of $\algebra$ are those of the form $\Phi(T)=\sum_{k \geq 0} a_k S^k +\sum_{k \geq 1} a_{-k} L^k$
for some sequence $(a_k)_k \in \ell_1(\Z,\K)$, which we shall denote $(a_k(T))_k$, and we have that
$\|\Phi(T)\|=\sum_{k \in \Z}|a_k(T)|$.
For simplification we shall denote $\Phi(T)=\sum_{k \in \Z} a_k S^k$ in the situation above, even
if $S$ is not formally invertible.
The map $\Phi$ has the property that $T-\Phi(T)$ is strictly singular for any $T \in L(X_S(\K))$, which allows to reduce most of the
study of operators on $X_S(\K)$ to operators in $\algebra$.

From this the authors of \cite{GM2} concentrate on the complex case, in which case $\ell_1(\Z)$ identifies with the Wiener algebra
$A(\T)$ of complex valued functions in $C(\T)$ whose Fourier series have absolutely summable coefficients.

We may use the complex version $X_S(\C)$ of $X_S$ to give a negative answer to the question of Pietsch in the complex case.
In order to be able to treat the real case as well we shall see that it is enough to replace $X_S(\C)$ by the complexification
of the real version of $X_S$, denoted $(X_S(\R))_\C$. A few comments are in order. Both $X_S(\C)$ and $(X_S(\R))_\C$ have  natural Schauder bases and contain two
canonical isometric
real subspaces $W$ and $iW$, where $W$ is the space generated by real linear combinations of elements of the basis.
While in the complexification $(X_S(\R))_\C$ these two form a direct sum, this is probably not the case inside $X_S(\C)$. Indeed the ``no shift" version
of the norm
of this space is the norm on Gowers-Maurey's HI space, which is known to be HI as a real space
( see the comments on p475 of \cite{F}) and therefore indecomposable as a real space,
and it is probable that similarly $W$ and $iW$ do not form a direct sum in $X_S(\C)$. This makes it more difficult to study real subspaces of
$X_S(\C)$ and suggests the use of $(X_S(\R))_\C$ instead. 

\

Consider 
the complexification ${X_S(\R)}_\C$. Note that it is  equipped with the complexification of the shift operator on $X_S(\R)$, which is just the shift operator
on ${X_S(\R)}_\C$ with its natural basis, and which we denote also $S$; therefore $S$ is a power bounded, isomorphic
embedding on the space, inducing an isomorphism with
its hyperplanes. Likewise the complexification of the left shift is power bounded.
By classical results about complexifications, operators on the space are of the form $T=A+iB$, where $A,B$ are real operators
(meaning that the formula $(A+iB)(x+iy)=Ax-By+i(Bx+Ay)$ holds); it follows that 
$$T(x+iy)=\sum_{k \in \Z} (a_k+ib_k) S^k(x+iy)+(V+iW)(x+iy)=\sum_{k \in \Z} \lambda_k S^k (x+iy)+(V+iW)(x+iy),$$
where the series $\lambda_k$ is absolutely summable in $\C$, the action of $S$ on the complex space $(X_S(\R))_\C$
is identified with the shift operator  $S$ there,
and where $V, W$ are strictly singular. By the results of Section 3 this is the same as saying
that $T-\sum_k \lambda_k S^k$ is strictly singular as a complex operator. Therefore we may also define
an algebra homorphism and projection map
(again called $\Phi$) from $L({X_S(\R)}_\C)$ to the algebra (again denoted $\algebra$) of operators of the form
$\Phi(T)=\sum a_k S^k$ for $(a_k)_k \in \ell_1(\Z,\C)$ denoted $(a_k(T))_k$.

\

Summing up,  in what follows, $X$ will denote either the complex version $X_S(\C)$ of the shift space,
or the complexification ${X_S(\R)}_\C$ of the real version of the shift space, and $\algebra$ and $\phi$
the corresponding algebra and map.

\

As in \cite{GM2}, $\Psi$ is the map defined from $L(X)$ to the Wiener algebra $A(\T)$ by
$$\Phi(T)=\sum_{k \in \Z} a_k S^k \Rightarrow \Psi(T)(e^{i\theta})=\sum_{k \in \Z} a_k e^{ki\theta}.$$
While in the case of $X=X_S(\C)$, $\Psi$ induces an isometric isomorphism between 
$\algebra$ and $A(\T)$ (\cite{GM2} Lemma 11), in the case of $X=(X_S(\R))_\C$ this is just
an isomorphism, whose constant depends on the equivalent norm chosen on $(X_S(\R))_\C$ (by
\cite{GM2} Lemma 11 in the real case).
This does not affect the rest of our computations.

We shall also denote by $\Phi$ the induced projection from $L(X^m,X^n)=M_{m,n}(L(X))$ onto $M_{m,n}(\algebra)$, i.e.
if $T=(T_{ij})_{i,j} \in M_{m,n}(L(X))$ then we define 
$$\Phi((T_{ij})_{i,j})=((\Phi(T_{ij})_{i,j})$$
and we note that
$$\Phi(T)=\sum_k A_k S^k$$ where $A_k=A_k(T) \in M_{m,n}(\C)$ is the matrix $(a_k(T_{ij}))_{i,j}$.

Likewise we define a map $\Psi$ from $L(X^m,X^n)$ to $M_{m,n}(A(\T))$  by the formula
$$\Psi(T)(e^{i\theta})=\sum_k A_k(T) e^{ik\theta}.$$

\

We shall make use of some notation and results of $K$-theory of Banach algebras. If $A$ is a Banach algebra,
then $M_\infty(A)$ denotes the set of $(n,n)$-matrices of elements of $A$ of arbitrary size, i.e.
$M_\infty(A)=\cup_n M_n(A)$ with the natural embeddings of $M_n(A)$ into $M_{n+1}(A)$.
Idempotents of $M_\infty(A)$ coincide with idempotents in one of the $M_n(A)$.
Among them $I_n$ denotes the identity on $M_n(A)$ (seen inside $M_\infty(A)$).
As usual $GL_n(A)$ 
 denotes the set of invertibles in $M_n(A)$.
If $A \subseteq L(X)$ is an algebra of operators on a space $X$ then $I_n$ will also be denoted
$Id_{X^n}$ or $I_{X^n}$.
Two idempotents $P,Q$ of $M_\infty(A)$ are {\em similar} if there exists
some $N \in \N$ and some $M$  in $GL_N(A)$, such that, denoting the natural
copy of $M$ inside $M_\infty(A)$ still by $M$, the relation
$$P=M^{-1}QM$$ holds. Note in particular that if $P$ and $Q$ are two similar
idempotents in $M_\infty(L(X))$ for some $X$ then the images $PX$ and $QX$ are
isomorphic.  Regarding the very basic results of $K$-theory we shall use, we
refer to \cite{blackadar} for background and \cite{maurey}
for a survey in a language familiar to Banach spaces specialists.

\subsection{Properties in the Wiener algebra $A(\T)$}

We recall classical or easy properties of the algebra $C(\T)$ of continuous complex functions on the complex circle $\T$, the Wiener algebra 
$A(\T)$ of functions in $C(\T)$ with absolutely summable Fourier series,
and their matrix algebras. They are certainly folklore but not always easy to find explicitely in the literature, so
we sometimes prefered to give a short proof rather than a too abstract or too general argument. We recall Wiener's Lemma \cite{wiener}: if
an element of $A(\T)$  is invertible in $C(\T)$ (i.e. does not vanish on $\T$), then its inverse belongs to $A(\T)$ as well.
 See \cite{maurey}, either Lemma 7.2 for a Banach space theoretic proof
 or the commentary after Proposition 2.2 for the classical proof. 

\begin{prop}\label{AAA} The following hold
 \begin{enumerate}
  \item An element $M$ of $M_n(C(\T))$ is invertible if and only if $det(M)$ is invertible in $C(\T)$
  \item If $M \in M_n(A(\T))$ is invertible in $M_n(C(\T))$ then it is invertible in $M_n(A(\T))$
  \item The set $GL_n(A(\T))$ of invertibles in $M_n(A(\T))$ is dense in $GL_n(C(\T))$
 \end{enumerate}

\end{prop}

\begin{proof} (a) follows from the cofactor formula in the abelian algebra $C(\T)$. (b) 
follows from the cofactor formula and the fact that $det(M)^{-1}$ belongs to
 $A(\T)$ by Wiener's Lemma. (c) since $A(\T)$ is dense in $C(\T)$, $M_n(A(\T))$ is dense in $M_n(C(\T))$. If $M$
 is invertible in $M_n(C(\T))$ then a close enough operator in $M_n(A(\T))$ will be invertible
 with respect to $M_n(C(\T))$ and therefore to $M_n(A(\T))$.
\end{proof}

\begin{lemma}\label{similarity} Two idempotents of $M_\infty(A(\T))$ which are similar in $M_\infty(C(\T))$ are similar in $M_\infty(A(\T))$.
\end{lemma}

\begin{proof} Let $P$ and $Q$ be such idempotents, and let $M$ be invertible in some  $GL_N(C(\T))$ such that
$Q=MPM^{-1}$. By Proposition \ref{AAA} (3) we may find a perturbation $M'$ of $M$ belonging to $GL_N(A(\T))$. Then
$Q'=M'P{M'}^{-1}$ is an idempotent of $M_N(A(\T))$ which is similar to $P$ in $M_N(A(\T))$, but also to $Q$ if $M'$
was chosen close enough to $M$.
Indeed it is a classical and immediate computation (valid in any Banach algebra) that $Q$ and $Q'$ are similar through the invertible $U=I-Q(Q'-Q)+(Q-Q')Q$
as soon as $Q'$ is
close enough to $Q$ in $M_N(C(\T))$ (see e.g. \cite{maurey} Lemma 9.2). Since $Q$ and $Q'$ belong to the algebra $M_N(A(\T))$,
$U$ is an invertible of $M_N(A(\T))$.
\end{proof}

\subsection{Complemented subspaces in powers of $X$}

\

Recall that $X$ is either $X_S(\C)$ or $X_S(\R)_\C$.
We now prove several results indicating how the rigidity properties of $X$ proved in \cite{GM2} carry over to its powers $X^n$.
As a first result and for clarity let us quickly repeat the ideas of \cite{GM2} to show that $X_S(\R)_\C$ also satisfies the
equivalence of Proposition \ref{GM}.

\begin{prop}\label{GMbis} The following are equivalent for an infinite dimensional subspace $Y$ of $X$
\begin{enumerate}
\item $Y$ is isomorphic to $X$
\item $Y$ is complemented in $X$
\item $Y$ is finite codimensional in $X$
\end{enumerate}
\end{prop} 

\begin{proof} (3) $\Rightarrow$ (2) is trivial, and (3) $\Rightarrow$ (1) is due to the existence of the shift operator $S$.
 (1) $\Rightarrow$ (3): if there is an embedding of $X$ into $X$, it is not infinitely singular,
 and it follows that it must be Fredholm. This can be seen as a consequence of a more general result, Proposition
 \ref{abc} (2)(3), whose proof follows below. (2) $\Rightarrow$ (3): If $P$ is a projection on $X$
 then $\Psi(P)$ is an idempotent in $A(\T)$, therefore it is either constantly $0$ or $1$,
 meaning that $\Phi(P)$ is either $I_X$ or $0$. Then either $P$ or $Id-P$ is a strictly singular
 projection
 and therefore has finite rank. So $Y=PX$ has finite codimension.
\end{proof}

\begin{lemma}\label{weird} Let $T \in L(X^m,X^n)$, for $m, n \in \N$. If  $(\alpha_i)_{i=1,\ldots,m} \in \C^n$ belongs to $Ker(\Psi(T)(t))$ for some $t \in \T$,
then the restriction of $T$ to the subspace
$\{(\alpha_1 x,\ldots, \alpha_m x), x \in X\}$ of $X^m$ is infinitely singular.
\end{lemma}

\begin{proof} This is a multidimensional version of Lemma 14 from \cite{GM2}. 
Recall that $\Psi(T)(t)=\sum_k A_k t^k$, where the $A_k$ are $(m,n)$ scalar matrices,
and $\Phi(T)=\sum_k A_k S^k$. 
We consider $$x_N:=x_N(t)=\frac{\log_2(1+N^2)}{N^2} \sum_{j=N^2}^{2N^2} t^{-j} e_j \in X,$$
which has norm at least $1$ by \cite{GM2} Theorem 5,
and prove that if $\alpha:=\begin{pmatrix} \alpha_1 \\ \ldots \\ \alpha_m \end{pmatrix} \in ker(\sum_k A_k t^k)$, i.e.
$\sum_k t^k A_k \alpha=0$, then
$$\Phi(T)(\alpha_1 x_N, \ldots, \alpha_m x_N)$$ is arbitrarily small. This will imply
that $\Phi(T)$ is infinitely singular and therefore $T$ as well, on the required subspace.

Take $N$ large enough
so that $\|\Phi(T)-U\|$ is less than some given $\epsilon>0$, with $U=\sum_{k=-N}^N A_k S^k$.
Then
$$\frac{N^2}{\log_2(1+N^2)} U(\alpha_1 x_N, \ldots, \alpha_m x_N)=\sum_{k=-N}^N \sum_{j=N^2}^{2N^2} t^j A_k S^{-k} (\alpha e_j) $$
$$=\sum_k \sum_{j=N^2}^{2N^2} t^j A_k  (\alpha e_{j+k})=\sum_{l=N^2-N}^{2N^2+N} t^l \sum_{l-2N^2 \leq k \leq l-N^2, -N \leq k \leq N} t^{k} A_k  (\alpha e_{l}))
$$
The sum inside is not zero only if for those $l$ such that $l-N^2 < N$ or $-N < l-2N^2$, therefore for at most $4N$ values of $l$,
and is uniformly bounded by absolute convergence of $\sum_k A_k$. Therefore $U(\alpha. x_N)$ is controled in norm by a multiple of $\log(N)/N$.
This concludes the proof.
\end{proof}

\begin{prop} \label{abc}
The following hold for $n \in \N$:
\begin{itemize}
\item[(1)]
Let $T \in L(X,X^n)$, written in blocks as
$T=\begin{pmatrix} T_1 \\ \vdots \\ T_n \end{pmatrix}.$
If  $\sum_i |\Psi(T_i)|^2$  vanishes on $\T$, then $T$ is infinitely singular. 
\item[(2)]
Let $T \in L(X^n)$. If $det(\Psi(T))$ vanishes on $\T$ then $T$ is infinitely singular.
\item[(3)] Let $T \in L(X^n)$. If 
$det(\Psi(T))$ does not vanish on $\T$ then $T$ is Fredholm.
\end{itemize}
\end{prop}

\begin{proof} (1) and (2) follow from Lemma \ref{weird}. (3) If $det(\Psi(T)(t)) \neq 0$ for all $t \in \T$ then $\Psi(T)$ is invertible
in $M_n(A(\T))$ by Proposition \ref{AAA} (1)(2). Let $U$ be an operator such that $\Psi(U)=\Psi(T)^{-1}$. 
 From $\Psi(TU)=Id$ and $\Psi(UT)=Id$ we deduce $TU-Id$ and $UT-Id$ are strictly singular, and therefore $T$ is Fredholm.
\end{proof}

\begin{cor}\label{primen} Operators on $X^n$ are either Fredholm or infinitely singular.
In particular the space $X^n$ is not isomorphic to its subspaces of infinite codimension.
\end{cor}

\

In the  next proposition we shall use the fundamental fact that the monoid of similarity classes of idempotents in $M_\infty(C(\T))$ is $\N$,
or equivalently, that the rank (i.e. for $A \in M_\infty(C(\T))$ the common rank of all matrices $A(t)$ for $t \in \T$),
is the only associated similarity invariant.  This is a consequence of the essential fact that $K_1(\C):=K_0(C_0(\T))$
identifies with the set of homotopy classes
of invertibles in $GL_n(\C)$ and therefore is $\{0\}$ by connexity of $GL_n(\C)$
(here $C_0(\T)$ denotes elements of $C(\T)$ which vanish in $0$), see for example \cite{blackadar} Theorem 8.22,
which also reformulates as the $K_0$-group of $C(\T)$ being equal to $\Z$, 
see for example \cite{blackadar}, Example 5.3.2 (c), or \cite{maurey}, Example 1 p49 or Examples
9.4.1.

\begin{prop}\label{multi} Let $n \in \N$. A complemented subspace of $X^n$ is isomorphic to $X^m$ for some $m \leq n$.
\end{prop}

\begin{proof} Let $P$ be a projection defined on $X^n$ and note that $\Phi(P)$ is also a projection,
which is a strictly singular perturbation of $P$. According to the Lemma on p49 of  \cite{maurey},
the map $P$ is therefore similar to a projection onto either some finite codimensional subspace
of $\Phi(P)X^n$, or $\Phi(P)X^n \oplus E$ where $E$ is finite dimensional.
Therefore $PX$ is a finite dimensional perturbation of $\Phi(P)X^n$ and since $X^m$ is isomorphic to its finite dimensional perturbations,
it is enough to prove the assertion for $\Phi(P)$. In other words we may assume that $P \in M_n(\algebra)$.

The image of $P$ through $\Psi$ is an idempotent of $M_n(A(\T))$ and in particular of $M_n(C(\T))$.
By the fact before the proposition,  $\Psi(P)$ 
is similar inside  $M_\infty(C(\T))$ to one of the canonical projections $I_m$ (i.e. the identity of $M_m(C(\T))$).
According to Lemma \ref{similarity}, it follows that $\Psi(P)$ is similar to $I_m$ inside $M_\infty(A(\T))$, i.e. 
$$\Psi(P)=M I_m M^{-1}$$ for some invertible $M$ in $M_N(A(\T))$ of appropriate dimension, and therefore the relation lifts to
$$P=U Id_{X^m} U^{-1}$$ for some invertible $U$ in $GL_N(\algebra)$ (seeing also $P$ and $Id_{X^m}$ as operators on $X^N$ in
the canonical way). It follows that $PX^n$ is isomorphic to $X^m$.
Finally $m \leq n$ as a consequence of Corollary \ref{primen}.\end{proof}

\

The proof of the above proposition implies the more technical result which follows:

\begin{lemma}\label{tecnico} If $P \in M_n(\algebra)$ is a projection on $X^n$ such that $PX$ is isomorphic to $X$, then
 there exist operators $U_1,\ldots,U_n, V_1,\ldots, V_n$ in $\algebra$ such that 
 $$PX^n=\{ (U_1 x, \ldots, U_n x), x \in X\}$$
 and such that $$U_1 V_1 + \cdots + U_n V_n=Id_X.$$
\end{lemma}

\begin{proof} By the above $P=U Id_X U^{-1}$ for some $U \in GL_N(\algebra)$ in the appropriate dimension $N$, but
it is easily checked that we may assume this dimension to be $n$ and therefore $U \in GL_n(\algebra)$.
 It follows
 that $P$ admits the matrix representation $$P=(U_i V_j)_{1 \leq i,j \leq n}$$
 with $\sum_i V_i U_i=Id_X$, where $(U_1,\ldots,U_n)$ is the first column of $U$ and $(V_1,\ldots,V_n)$ the
 first line of $U^{-1}$ and  therefore these operators belong to $\algebra$.
 Note also that $Id_X=U_1 V_1 + \cdots + U_n V_n$ since $\algebra$ is abelian.
  We have the formula $$P(x_1,\ldots,x_n)=(U_1 z, \ldots, U_n z)$$
 where $z=\sum_i V_i x_i$  and since $\sum_i V_i U_i=Id_X$, $z$
 takes all possible values in $X$. Therefore
 $$PX^n=\{ (U_1 x, \ldots, U_n x), x \in X\}.$$
 \end{proof}

\

A $1$-dimensional subspace of $\C^n$ generated by a vector $a$ is complemented
by the orthogonal projection $p(v)=\frac{<a,v>}{\|a\|^2} a$. 
In the next lemma we show how a similar result holds in $X^n$ for operators
in $M_n(\algebra)$. 
By  ${\rm diag}(M)$ we shall denote the diagonal block matrix operator on $X^n$ with 
 $M \in L(X)$ on the diagonal.
 For arbitrary $W \in \algebra$, we denote by $\overline{W}$ the 
 operator in $\algebra$ such that $\Psi(\overline{W})=\overline{\Psi(W)}$. That is,
 if $W=\sum_{n \in \Z} a_n S^n $, then $\overline{W}=\sum_{n \in \Z} \overline{a_{-n}} S^n$.
 We extend this definition in an obvious way to elements of $M_{m,n}(\algebra)$.
 Finally if $T \in M_{n,1}(A)$, then $T^t  \in M_{1,n}(A)$ denotes the transposed matrix of $T$.

\begin{lemma}\label{projection} Assume $T \in M_{1,n}(\algebra)$ is finitely singular from $X$ to $X^n$.
Then $TX$ is complemented in $X^n$ by the projection $P=B {\rm diag}(A^{-1})$,
where $A:=\overline{T}^t.T
\in \algebra$ and
$B=T.\overline{T}^t \in M_n(\algebra)$.
\end{lemma}

\begin{proof}
 Let us see $T$ as a column
$$T=\begin{pmatrix} T_1 \\ \vdots \\ T_n \end{pmatrix},$$
where each $T_i$ is an operator in $\algebra$.
Since $T$ is finitely singular, 
 $\Psi(A)=\sum_j |\Psi(T_i)|^2$ does not vanish on $\T$, by Proposition \ref{abc} (a). 
 So it is invertible (in $A(\T)$ by Wiener's lemma), and so $A$ is invertible in  $\algebra$.
The map takes values in $TX$ and we claim that $PT=T$, implying that $P$ is a projection onto $TX$.
The claim follows from the computation
 (using that $\algebra$ is abelian)
$$PT =T\overline{T}^t {\rm diag}(A^{-1})T
=T\overline{T}^t (A^{-1} T_i)_i=
T \sum_i (\overline{T_i} A^{-1} T_i)=T AA^{-1}=
T.$$
\end{proof}

We now can prove the main technical result of this section.

\begin{prop}\label{important} Assume $m, n \in \N$. Let $Y$ be an infinite codimensional subspace of $X$.
Then there is no isomorphism between a complemented subspace of $X^m$ and a
 subspace of $Y^n$.
\end{prop}
 
 \begin{proof} Assume there is such an isomorphism.
 By Proposition \ref{multi}, it follows that there exists  an isomorphic embedding $R$ of $X$ into $Y^n$.
We denote $T=\Phi(R)$, and since $R-T$ is strictly singular, we note that $T$ is finitely singular.
So by Lemma \ref{projection}, $P=T\overline{T}^t {\rm diag}(A^{-1})$ is a projection onto $TX$,
where $A:=\overline{T}^tT$.

Let $U_i, V_i$ be given for $P$ by Lemma \ref{tecnico}. Therefore, and letting $s:=T-R$,
$$\begin{pmatrix} U_1 \\ \vdots \\ U_n \end{pmatrix}
=P \begin{pmatrix} U_1 \\ \vdots \\ U_n \end{pmatrix}
=T \overline{T}^t {\rm diag}(A^{-1}) \begin{pmatrix} U_1 \\ \vdots \\ U_n \end{pmatrix}
=s \overline{T}^t {\rm diag}(A^{-1}) \begin{pmatrix} U_1 \\ \vdots \\ U_n \end{pmatrix}
+R  \overline{T}^t {\rm diag}(A^{-1}) \begin{pmatrix} U_1 \\ \vdots \\ U_n \end{pmatrix}
.$$

Since $s$ is strictly singular, the operator $\begin{pmatrix} U_1 \\ \vdots \\ U_n \end{pmatrix}$ is therefore the sum of a strictly singular operator
$\begin{pmatrix} s_1 \\ \vdots \\ s_n \end{pmatrix}$ and of an operator with values in $Y^n$,
which implies that $U_i-s_i$ takes values in $Y$ for $i=1,\ldots, n$.
Then the operator $\sum_i (U_i-s_i)V_i$ takes values in $Y$ and on the other hand
it is equal to $\sum_i U_i V_i -\sum_i s_i V_i=Id_X-\sum_i s_i V_i$. We would therefore
obtain a strictly singular perturbation of the identity with values in an infinite
codimensional subspace of $X$, a contradiction with the stability properties
of the Fredholm class.
\end{proof}

\
 
 We finally arrive to the objective of this section, the {\em Proof of Proposition \ref{mainprop}}: if
 $Y$ is an infinite codimensional subspace of $X_S$ (real or complex), then there is no isomorphism between a complemented subspace of $X_S^m$ and a
 subspace of $Y^n$.

\begin{proof} In the complex case this is just Proposition \ref{important} for $X=X_S(\C)$.
 In the real case, if $Z$ is a complemented subspace of $X_S^m$ isomorphic to a subspace
 of $Y^n$, then $Z_\C$ is a complemented subspace of $(X_S)_\C^m$ isomorphic to a subspace
 of $Y_\C^n$, and therefore $Z_\C$ (and $Z$) must be finite dimensional by
 Proposition \ref{important} in the case $X=(X_S(\R))_\C$.
\end{proof}

\section{Comments and problems}

Our results leave open the following new version of \cite{P} Problem 2.2.8.

\begin{quest}\label{q}
For which space ideals $A$ does 
there exist a largest operator ideal $U$ with $A = \Space(U)$?
\end{quest}

Recall that a space ideal is a class of spaces containing the finite dimensional ones and stable
under taking direct sums and complemented subspaces. For any ideal $\Ideal$,
$\Space(\Ideal)$ is a space ideal, \cite{P} Theorem 2.1.3, and conversely a space ideal $A$ always coincides with
$\Space(\Op(A))$, \cite{P} Theorem 2.2.5.
And our main result is that the answer to Question \ref{q} is
negative for the space ideal $F$ of finite dimensional spaces.

\

The "next" natural example seems to be the space ideal $H$ of separable
spaces isomorphic to a Hilbert space (finite or infinite dimensional). So
it seems natural to ask:

\begin{quest}
 Does there exist a largest operator ideal $\Ideal$ with $H=\Space(\Ideal)$?
\end{quest}

Pietsch (\cite{P} 4.3.7) considered a natural candidate for the largest operator ideal $V$ such that
$\Space(V)=\Space(\Ideal)$; namely $\Ideal^{rad}$, defined by
$$S \in \Ideal^{rad}(X,Y) \Leftrightarrow \forall L \in L(Y,X),  W(I_X-LS)=I_X-T$$
for some $W \in L(X)$ and some $T \in {\Ideal}(X)$ (\cite{P} Definition 4.3.1).
In particular $F^{rad}=\In$. In particular he proved
that $\Space(\Ideal)=\Space(\Ideal^{rad})$ always holds, \cite{P} Proposition 4.3.6.
Here it does not even seem to be clear (to the author) what the ideal $H^{rad}$ is.



\

It may be amusing to observe that it follows from Proposition \ref{multi} that the class $A$ of spaces
isomorphic to some power of $X_S$ is a space ideal. Therefore $A={\rm Space}(\Opp(X_S))$ by \cite{P} Theorem 2.2.5.

\
 
Some natural comments and questions about examples from the first part of the paper are also included below.

\begin{quest} Are the spaces $Z_2(\alpha)$ and $\overline{Z_2(\alpha)}$ from \cite{K} projectively incomparable for $\alpha \neq 0$? essentially incomparable?
\end{quest}

Ferenczi-Galego \cite{FG1} prove that if a space is essentially incomparable with is conjugate,
then it does not contain a complemented subspace with an unconditional basis. 
For $Z_2(\alpha)$ (more generally, for twisted Hilbert spaces),  by Kalton \cite{KIMJ},
  a  complemented subspace with an unconditional basis would have to be hilbertian. 
  We do not know whether $Z_2(\alpha)$ contains a complemented Hilbertian copy (for $Z_2$ this
  is impossible, by \cite{KP} Corollary 6.7).
It may be worth pointing out that the above result from \cite{FG1} actually holds (with the same proof) for projective incomparability:

\begin{prop} If a complex space $X$ is projectively incomparable with its conjugate, 
then it does not contain a complemented subspace with an unconditional basis.
\end{prop}

\begin{proof} If $Y$ is a subspace of $X$ with an unconditional basis $(e_n)$, then
$\overline{Y}$ is a subspace of $\overline{X}$ which is isomorphic to $Y$ by the map
$\sum_i \lambda_i e_i \mapsto \sum_i \overline{\lambda_i} e_i$. If $Y$ is complemented in $X$
then $\overline{Y}$ is complemented in $\overline{X}$. \end{proof}

\

{\em Acknowledgements:} We warmly thank Manuel Gonz\'alez for drawing our attention to the theory of proper
operator ideals and the questions of Pietsch, for suggesting to compare the real and
complex versions of classical ideals, as well as for useful comments on a first draft of this paper.


\begin{thebibliography}{99}
\bibitem{AG} P. Aiena, M. Gonz\'alez, {\em Examples of improjective operators}, Math. Z. 233 (2000), 471--479.

\bibitem{anisca} R. Anisca, {\em  Subspaces of $L_p$ with more than one complex structure}, Proc. Amer. Math.
Soc. (131) 9 (2003), 2819--2829.
\bibitem{AFM} R. Anisca, V.Ferenczi, Y. Moreno, {\em On the classification of complex structures and positions in Banach
spaces}, Journal of Functional Analysis 272 (2017) , 3845--3868.

\bibitem{blackadar} B. Blackadar, {\em K-theory for operator algebras}, MSRI Publications 5, Springer Verlag
1986.
\bibitem{Bourgain} J. Bourgain, {\em Real isomorphic Banach spaces need not be complex isomorphic}, Proc.
Amer. Math. Soc. (96) 2 (1986), 221--226.
\bibitem{Burlando} L. Burlando, {\em On subspaces of direct sums of infinite sequences of Banach
spaces}. Atti Accad. Ligure Sci. Lett. 46 (1989), 96--105.
\bibitem{mestradocuellar} W. Cuellar Carrera, {\em Um espa\c co de Banach n\~ao isomorfo a seu conjugado complexo (A
Banach space not isomorphic to its complex conjugate)}, Master's Dissertation,
Universidade de S\~ao Paulo, 2011.
\bibitem{cuellar} W. Cuellar Carrera, {\em A Banach space with a countable infinite number of complex structures},
J. Funct. Anal. 267 (5) (2014), 1462--1487.  
\bibitem{C} W. Cuellar Carrera, {\em Complex structures on Banach spaces with a subsymmetric basis}, 
J. Math. Anal. App. 440 (2) (2016), 624--635.
\bibitem{dieudonne} J. Dieudonn\'e, {\em Complex structures on real Banach spaces}. Proc. Amer. Math. Soc.
(3) 1 (1952), 162--164.
\bibitem{F} V. Ferenczi,  {\em Uniqueness of complex structure and real hereditarily indecomposable Banach spaces}, Adv.
Math. 213 (1) (2007) 462--488.
\bibitem{FG1} V. Ferenczi, E.M. Galego, {\em Even infinite-dimensional real Banach spaces}, Journal of Functional Analysis 253 (2007), 534--549.
\bibitem{FG} V. Ferenczi, E.M. Galego, {\em Countable groups of isometries}, Transactions of the American Mathematical Society,  
Volume 362, Number 8 (2010), 4385--4431.
\bibitem{G} M. Gonz\'alez, {\em On essentially incomparable Banach spaces}, Math Z. 215 (1994), 621-629.
\bibitem{GH} M. Gonz\'alez, J. Herrera,
{\em Decompositions for real Banach spaces with
small spaces of operators}, Studia Mathematica 183 (1) (2007), 1--14.
\bibitem{GM} W.T. Gowers, B. Maurey, {\em The unconditional basic sequence problem}. J. Amer. Math.
Soc. (6) 4 (1993), 851--874.
\bibitem{GM2} W.T. Gowers, B. Maurey, {\em Banach spaces with small spaces of operators}. Math Ann. (307) (1997),
543--568.
\bibitem{bence} 
B. Horv\'{a}th, {\em A Banach space whose algebra of operators 
is Dedekind-finite but it does not have stable rank one}, Banach Algebras 
and Applications, de Gruyter (editor: M.~Filali) (2020), 165--176. 


\bibitem{K} N.J. Kalton, {\em An elementary example of a Banach space not isomorphic to its complex conjugate}, Canad. Math. Bull. 38 (2) (1995) 218--222.
\bibitem{KIMJ} N.J. Kalton, {\em Twisted Hilbert spaces and unconditional structure},  Journal of the Inst. of Math. of Jussieu (2003) 2(3), 401--408.
\bibitem{KP} N.J. Kalton, N. T. Peck, {\em Twisted sums of sequence spaces and the three space problem},  Transactions of the American Mathematical Society, 
 Volume 255 (1979), 1--30.
 \bibitem{koszmider} P. Koszmider, M. Martin, J. Meri {\em Extremely non-complex $C(K)$-spaces}, J. Math.
 Anal. Appl. 350 (2) (2009), 601--615.
 \bibitem{LS} A. Lebow, M. Schechter, {\em Semigroups of operators and measures of noncompactness}.
J. Funct. Anal. 7 (1971), 1--26.
\bibitem{LT} J. Lindenstrauss, L. Tzafriri, {\em Classical Banach spaces}, Springer-Verlag, New York, Heidelberg, Berlin, 1979.
\bibitem{maurey} B. Maurey, {\em Operator theory and exotic Banach spaces}. 1996. Notes available on the
author's webpage
\bibitem{maureyhandbook} B. Maurey, {\em Banach spaces with few operators}, Handbook of the Geometry
of Banach Spaces, Vol. 2, 2003, 1247--1297, Edited by William B. Johnson and Joram Lindenstrauss, Elsevier.
\bibitem{P} A. Pietsch, {\em Operator ideals}, North-Holland, Amsterdam, New York, Oxford, 1980.
\bibitem{szarek} S. Szarek {\em A super reflexive Banach space which does not admit complex structure}.
Proc. Amer. Math. Soc. (97) 3 (1986), 437--444.
\bibitem{Tarafdar} E. Tarafdar, {\em Improjective operators and ideals in a category of Banach spaces}. J. Austral.
Math. Soc. 14 (1972), 274--292
\bibitem{wiener} N. Wiener, {\em On the representation of functions by trigonometrical integrals} (1926)
Mathematische Zeitschrift, 24(1) 575--616.


\end{thebibliography}
\end{document}